\documentclass[12pt]{article}

\usepackage{amsmath,amsfonts,amssymb,latexsym,amsthm,verbatim,url,color}

 \theoremstyle{plain}
  \newtheorem{theorem}{Theorem}[section]
  \newtheorem{corollary}[theorem]{Corollary}
  \newtheorem{proposition}[theorem]{Proposition}
  \newtheorem{lemma}[theorem]{Lemma}  
  
  \newtheorem{remark}[theorem]{Remark}
  \newtheorem{example}[theorem]{Example}

\title{On strong stationary times and approximation of Markov chain hitting times by geometric sums}
\author{Fraser Daly\footnote{Department of Actuarial Mathematics and Statistics and the Maxwell Institute for Mathematical Sciences, Heriot-Watt University, Edinburgh, EH14 4AS UK. Email: \texttt{f.daly@hw.ac.uk}}}
\date{\today}

\begin{document}

\maketitle

\begin{abstract}
Consider a discrete time, ergodic Markov chain with finite state space which is started from stationarity. Fill and Lyzinski (2014) showed that, in some cases, the hitting time for a given state may be represented as a sum of a geometric number of IID random variables. We extend this result by giving explicit bounds on the distance between any such hitting time and an appropriately chosen geometric sum, along with other related approximations. The compounding random variable in our approximating geometric sum is a strong stationary time for the underlying Markov chain; we also discuss the approximation and construction of this distribution.
\end{abstract}
\vspace{12pt}


\vspace{12pt}


\section{Preliminaries} \label{sec:intro}

Throughout, we let $X=\{X_t:t\geq0\}$ be an irreducible, ergodic, discrete time Markov chain with finite state space $S$, transition matrix $P$, and stationary distribution $\pi$.  Our primary interest in this note is in the time it takes $X$, when initialised in the stationary regime, to hit a given state $j\in S$.  Many structural results are known about such hitting times.  In particular, it is known that, under certain conditions, this hitting time may be expressed as a geometric sum of independent and identically distributed (IID) random variables; see Theorem 4.2 of Fill and Lyzinski \cite{fl14} for a precise statement.  Our main goal in this note is to consider the approximation of such a hitting time by a geometric sum.  Motivated by the results of \cite{fl14}, an appropriate choice for the compounding distribution in the approximating geometric sum is that of a strong stationary time for the underlying Markov chain $X$ with an appropriate initial distribution.  Recall that a strong stationary time, $T$, for $X$ is a randomized stopping time such that $X_T\sim\pi$, and $X_T$ is independent of $T$.  We refer the reader to \cite{ad87} and \cite{df90} for background on strong stationary times, and to \cite{f09}, \cite{ls12} and references therein for more recent developments.   

We use the remainder of this section to introduce some necessary background material on strong stationary times, and to provide an explicit link between geometric sums of strong stationary times and Markov chain hitting times.  Section \ref{sec:approx} then contains our main approximation results, and a comparison of the present results with those of Daly \cite{d10}, who has previously considered the approximation of Markov chain hitting times by geometric sums.  Note, however, that the approximating geometric sum chosen in \cite{d10} is very different to that used here, and that the results of \cite{d10} have some deficiencies which the present results remedy (for example, the bounds in \cite{d10} offer no guarantee of sharpness).  We also note that the approximating geometric sum we choose here stochastically dominates our hitting time, and explore implications of this in Section \ref{sec:approx}.  Finally, in Section \ref{sec:greedy} we conclude with some remarks about the construction of the strong stationary times needed for our approximating geometric sum.   

Throughout this note we let the $t$-step transition probabilities of our Markov chain $X$ be denoted by $P^t(k,l)=\mathbb{P}(X_t=l|X_0=k)$ for $k,l\in S$.  For a state $j\in S$, we define the distribution $\pi^{(j)}$ as the stationary distribution $\pi$ restricted to the states other than $j$: $\pi^{(j)}=\frac{\pi-\pi_j\delta_j}{1-\pi_j}$, where $\delta_j$ is the Dirac delta. We will write $N\sim\mbox{Geom}(p)$ if $\mathbb{P}(N=k)=p(1-p)^k$ for $k=0,1,2,\ldots$, and let `$\stackrel{d}{=}$' denote equality in distribution

A strong stationary time $T$ will be called `fastest' if it is stochastically smaller than any other strong stationary time.  Note that Proposition 3.2 of Aldous and Diaconis \cite{ad87} guarantees the existence of such a fastest strong stationary time for any Markov chain $X$ of the type we consider here.

The tails of strong stationary times are closely related to separation, which may be used to estimate how far $X_t$ is from stationarity.  The separation at time $t$ is defined by 
\[
s(t)=1-\inf_{s\in S}\left\{\frac{\mathbb{P}(X_t=s)}{\pi_s}\right\}\,.
\]
As with strong stationary times, for convenience we suppress dependence on the initial distribution of $X$ in the notation. 

The following lemma (see Proposition 3.2 of \cite{ad87}) says that tails of a fastest strong stationary time achieve separation.  
\begin{lemma}[\cite{ad87}] \label{lem:sep}
There exists a strong stationary time $T$ for $X$ such that $\mathbb{P}(T>t)=s(t)$ for all $t\geq0$. 
\end{lemma}
We now give a lemma which collects facts about strong stationary times from the literature which we will need in what follows.  Although these facts are known, we give a short proof for the benefit of the reader, and to motivate the work that follows (in particular, the proof of Theorem \ref{thm:GeomSum}).   
\begin{lemma} \label{lem:beforeT}
Let $j\in S$ be such that $\pi_y\mathbb{P}(X_t=j)\leq\pi_j\mathbb{P}(X_t=y)$, for all $t\geq0$ and $y\in S$, when $X_0\sim\pi^{(j)}$.  Then $P^t(j,j)$ is decreasing in $t$ and, letting the random variable $T^{(j)}$ be defined by
\[
\mathbb{P}(T^{(j)}>t)=\frac{P^t(j,j)-\pi_j}{1-\pi_j}\,,\qquad t=0,1,\ldots\,,
\]
$T^{(j)}$ has the distribution of a fastest strong stationary time for $X$ with initial distribution $X_0\sim\pi^{(j)}$.  Furthermore, $\mathbb{P}(T^{(j)}>t,X_t=j)=0$ for all $t\geq0$.
\end{lemma}
\begin{proof}
Under the given condition on state $j$, when $X_0\sim\pi^{(j)}$,
\[
s(t)=1-\frac{\mathbb{P}(X_t=j)}{\pi_j}=1-\frac{\sum_{s\not=j}P^t(s,j)\pi_s}{\pi_j(1-\pi_j)}
=1-\frac{\pi_j-P^t(j,j)\pi_j}{\pi_j(1-\pi_j)}=\frac{P^t(j,j)-\pi_j}{1-\pi_j}\,.
\]
Since $s(t)$ is known to be decreasing in $t$ (see, for example, Chapter 9 of \cite{af}), we have that $P^t(j,j)$ is decreasing.  By Lemma \ref{lem:sep}, $T^{(j)}$ as defined has the distribution of a fastest strong stationary time.
Finally,
\begin{multline*}
\mathbb{P}(T^{(j)}>t,X_t=j) = \mathbb{P}(X_t=j)-\mathbb{P}(T^{(j)}\leq t,X_t=j)
= \mathbb{P}(X_t=j)-\mathbb{P}(T^{(j)}\leq t)\pi_j\\
= \pi_j\left(1-s(t)-\mathbb{P}(T^{(j)}\leq t)\right)
= 0\,,
\end{multline*}
where the second equality follows from Lemma 6.9 of \cite{lpw09}, the third by the assumption made on state $j$, and the final equality follows from Lemma \ref{lem:sep}.
\end{proof}
Using Lemma \ref{lem:beforeT}, Theorem \ref{thm:GeomSum} below gives an explicit link between strong stationary times and geometric sums.  Its conclusion is the same as that of Theorem 4.2 of Fill and Lyzinski \cite{fl14}, though the stated conditions are somewhat stronger (see Remark \ref{rem:FL} below), and the proof uses different techniques.  We give it here to motivate the main results of this paper that will follow in Section \ref{sec:approx}.
\begin{theorem} \label{thm:GeomSum}
Let $W_j=\inf\{t\geq0 : X_t=j\}$, the first time the Markov chain $X$ reaches the state $j$, where we assume the initial distribution $X_0\sim\pi$.  Assume that the state $j$ satisfies $\pi_y\mathbb{P}(X_t=j)\leq\pi_j\mathbb{P}(X_t=y)$, for all $t\geq0$ and $y\in S$, when $X_0\sim\pi^{(j)}$.  Let the random variable $T^{(j)}$ be defined as in Lemma \ref{lem:beforeT}.
Then $W_j\stackrel{d}{=}T^{(j)}_1+\cdots+T^{(j)}_N$, where $N\sim\mbox{Geom}(\pi_j)$ and $T^{(j)},T^{(j)}_1,T^{(j)}_2,\ldots$ are IID.
\end{theorem}  
\begin{proof}
Using the characterization of geometric sums presented in Section 2 of Daly \cite{d10}, and the fact that $\mathbb{P}(W_j=0)=\pi_j$, the conclusion follows on showing that $W_j+T^{(j)}\stackrel{d}{=}(W_j|W_j>0)$.  

From Lemma \ref{lem:beforeT}, $T^{(j)}$ has the distribution of a fastest strong stationary time for $X$ with initial distribution $X_0\sim\pi^{(j)}$.  Note that we may construct $(W_j|W_j>0)$ as the first hitting time of the state $j$ in a copy of the Markov chain $X$ started according to this initial distribution.  Also, by Lemma \ref{lem:beforeT}, the probability that this Markov chain visits state $j$ before time $T^{(j)}$ is zero, at which point we have achieved stationarity, and the subsequent time needed before it visits state $j$ has the same distribution as the original hitting time $W_j$ from stationarity.
\end{proof}
\begin{remark} \label{rem:FL}
\emph{Theorem 4.2 of Fill and Lyzinski \cite{fl14} also gives conditions under which the hitting time $W_j$ is distributed as a geometric sum.  Their result is stronger, as they work under the assumption that $P^t(j,j)$ is decreasing in $t$ in the place of our assumption that $\pi_y\mathbb{P}(X_t=j)\leq\pi_j\mathbb{P}(X_t=y)$ for all $t$ and $y$.  Lemma \ref{lem:beforeT} shows that our assumption is more restrictive than that of Fill and Lyzinski.}
\end{remark}

One of our aims in this note is to find an approximate version of Theorem \ref{thm:GeomSum}, in the sense that we wish to quantify explicitly the departure of the distribution of $W_j$ from that of a geometric sum, as well as deriving other bounds related to this approximation.  We will do this in Section \ref{sec:approx} below.  

\section{Approximation of Markov chain hitting times} \label{sec:approx}

Let $W_j=\inf\{t\geq0 : X_t=j\}$ be the Markov chain hitting time defined in Theorem \ref{thm:GeomSum}, and let $T^{(j)}$ be a strong stationary time for the Markov chain $X$ with initial distribution $X_0\sim\pi^{(j)}$.  Our main gaol in this section is to consider the approximation of $W_j$ by the geometric sum $U=T^{(j)}_1+\cdots+T^{(j)}_N$, where $N\sim\mbox{Geom}(\pi_j)$ and $T^{(j)},T^{(j)}_1,T^{(j)}_2,\ldots$ are IID.  We firstly do this by bounding the total variation distance $d_{TV}(\mathcal{L}(W_j),\mathcal{L}(U))$, defined by
\[
d_{TV}(\mathcal{L}(W_j),\mathcal{L}(U))=\sup_{A\subseteq\mathbb{Z}^+}|\mathbb{P}(W_j\in A)-\mathbb{P}(U\in A)|=\inf_{(W_j,U)}\mathbb{P}(W_j\not=U)\,,
\]
where the infimum is taken over all couplings of $W_j$ and $U$. 
Following the work of Daly \cite{d10}, to derive such a bound we construct an integer-valued random variable $Y$ (which may depend on $T^{(j)}$) such that $Y+T^{(j)}\geq0$ and
\[
Y+T^{(j)}\stackrel{d}{=}(W_j|W_j>0)\,.
\]   
Proposition 3.2 of \cite{d10} then gives us that
\[
d_{TV}(\mathcal{L}(W_j),\mathcal{L}(U))\leq\frac{1-\pi_j}{\pi_j}d_{TV}(\mathcal{L}(W_j),\mathcal{L}(Y))\,.
\]

\begin{remark} \label{rem:bound}
\emph{The results of \cite{d10} are stated under the assumption that the support of $T^{(j)}$ is bounded.  This assumption may be removed using by making suitable adjustments to the proof of Theorem 2.1 of that paper.  Specifically, we need analogues of Lemmas 4.1 and 4.2 of \cite{d10} which may be applied when the functions involved are not polynomials of bounded degree.  We note that a suitable analogue of Lemma 4.1 follows from Rouch\'e's theorem, and that Lemma 4.2 may be generalised by applying the residue theorem and a suitable change of variable.  We note that this allows us to remove the assumption that the compounding random variable has bounded support from all the results of \cite{d10}, including the approximation results for Markov chain hitting times.}
\end{remark}

Write $\widetilde{W_j}=(W_j|W_j>0)$.  As in Section \ref{sec:intro}, we may construct $\widetilde{W_j}$ as the time of the first visit of the Markov chain $X$ to the state $j$ when initialized with distribution $X_0\sim\pi^{(j)}$.  This hitting time may come either before the strong stationary time $T^{(j)}$ for this chain, or not.  On the event that $\widetilde{W_j}<T^{(j)}$, we set $Y=\widetilde{W_j}-T^{(j)}$; otherwise we have achieved stationarity at time $T^{(j)}$, and the remaining time until we reach state $j$ is distributed as $W_j$, and we may set $Y=W_j$.  Here we construct $W_j$ using a Bernoulli random variable $\xi\sim\mbox{Be}(\pi_j)$ (independent of all else) with $\mathbb{P}(\xi=1)=1-\mathbb{P}(\xi=0)=\pi_j$, and set $W_j=0$ if $\xi=1$; otherwise we set $W_j=\widetilde{W_j}$.

We therefore have that
\[
d_{TV}(\mathcal{L}(W_j),\mathcal{L}(Y))\leq\mathbb{P}(\widetilde{W_j}<T^{(j)})\leq\sum_{t=1}^\infty\mathbb{P}(X_t=j,T^{(j)}>t)\,.
\]
Evaluating this bound, (using the fact that $T^{(j)}$ is a strong stationary time) we have 
\begin{align*}
d_{TV}(\mathcal{L}(W_j),\mathcal{L}(Y))&\leq\sum_{t=1}^\infty\left[\mathbb{P}(X_t=j)-\mathbb{P}(T^{(j)}\leq t,X_t=j)\right]\\
&=\pi_j\sum_{t=1}^\infty\left[\frac{1-P^t(j,j)}{1-\pi_j}-\mathbb{P}(T^{(j)}\leq t)\right]\\
&=\pi_j\sum_{t=1}^\infty\left[\mathbb{P}(T^{(j)}>t)-\frac{P^t(j,j)-\pi_j}{1-\pi_j}\right]\\
&=\pi_j\mathbb{E}T^{(j)}-\frac{\pi_j^2}{1-\pi_j}\mathbb{E}W_j\,,
\end{align*}
where the final equality follows from the identity $\pi_j\mathbb{E}W_j=\sum_{t=0}^\infty\left[P^t(j,j)-\pi_j\right]$ given in Proposition 10.19 of \cite{lpw09}. We then arrive at the following.
\begin{theorem} \label{thm:TVbd}
Let $W_j=\inf\{t\geq0 : X_t=j\}$, where we assume that $X_0\sim\pi$.  Let $T^{(j)}$ be a strong stationary time (independent of $W_j$) for the Markov chain $X$ with initial distribution $X_0\sim\pi^{(j)}$.  Let $U=T^{(j)}_1+\cdots+T^{(j)}_N$, where $N\sim\mbox{Geom}(\pi_j)$ and $T^{(j)},T^{(j)}_1,T^{(j)}_2,\ldots$ are IID.  Then
\begin{equation} \label{eq:TVdb}
d_{TV}(\mathcal{L}(W_j),\mathcal{L}(U))\leq
(1-\pi_j)\mathbb{E}T^{(j)}-\pi_j\mathbb{E}W_j
\,.
\end{equation}
\end{theorem}
Note that Theorem \ref{thm:GeomSum} above follows as an immediate corollary of this result: if $j$ is such that $\pi_y\mathbb{P}(X_t=j)\leq\pi_j\mathbb{P}(X_t=y)$ for all $t$ and $y$, then Lemma \ref{lem:beforeT} gives us that we may take the strong stationary time $T^{(j)}$ to have distribution 
\[
\mathbb{P}(T^{(j)}>t)=\frac{P^t(j,j)-\pi_j}{1-\pi_j}\,,\qquad t=0,1,\ldots\,,
\]
and so the right-hand side of (\ref{eq:TVdb}) is zero (again, using Proposition 10.19 of \cite{lpw09}).

Daly \cite{d10} also treats the problem of approximation of the random variable $W_j$ by a geometric sum (with the same choice of geometric random variable $N$, but a different choice of compounding random variable).  Daly made his choice of compounding distribution to reflect (in a certain sense) the minimum time needed for the Markov chain to make a jump to the state $j$, while we have chosen $T^{(j)}$ to reflect the time needed to achieve stationarity.  This gives us a fundamentally different approximating geometric sum, and a result which allows us to derive bounds which reflect cases in which the hitting time $W_j$ is exactly distributed as a geometric sum (as in Theorem \ref{thm:GeomSum}).  There are no such known results corresponding to the approximation for Markov chain hitting times given in \cite{d10}.  Some illustration of this is provided by Example \ref{eg:2State} below.
\begin{example} \label{eg:2State}
\emph{Let $X$ have state space $\{0,1\}$ and transition matrix
\[
P=\begin{pmatrix}
1/2 & 1/2\\
1/2-\delta & 1/2+\delta
\end{pmatrix}\,,
\]
for some $0\leq\delta<1/2$, so that $\pi_0=\frac{1-2\delta}{2(1-\delta)}$, $\pi_1=\frac{1}{2(1-\delta)}$, and
\[
P^t=\frac{1}{2(1-\delta)}
\begin{pmatrix}
1-2\delta+\delta^t & 1-\delta^t\\
1-2\delta-\delta^t(1-2\delta) & 1+\delta^t(1-2\delta)
\end{pmatrix}\,.
\]
We let $W_1=\inf\{t\geq0 : X_t=1\}$.  Note that $P^t(1,1)$ is decreasing in $t$, so Theorem 4.2 of Fill and Lyzinski \cite{fl14} gives that $W_1\stackrel{d}{=}T^{(1)}_1+\cdots+T^{(1)}_N$, where $N\sim\mbox{Geom}(\pi_1)$, and $T^{(1)},T^{(1)}_1,T^{(1)}_2,\ldots$ are IID with $\mathbb{P}(T^{(1)}>t)=\delta^t$.  As one would expect, it is easily checked that the upper bound of Theorem \ref{thm:TVbd} is zero, also reflecting that fact that $W_1$ is distributed as a geometric sum.  However, Theorem 3.1 of Daly \cite{d10}, which also considers approximation of Markov chain hitting times by geometric sums, gives only the bound
\[
d_{TV}(\mathcal{L}(W_1),\mathcal{L}(N))\leq\frac{\delta(1-2\delta)}{2(1-\delta)^2}\,,
\]
i.e., the approximating geometric sum chosen by that result is the geometric random variable $N$, and that theorem does not reflect the fact that $W_1$ is itself distributed as a geometric sum.}
\end{example}

In Theorem \ref{thm:ord} below, we note that the approximating geometric sum $U$ is stochastically larger than the hitting time $W_j$.  We use `$\geq_{st}$' to denote the usual stochastic ordering, defined for random variables $Y$ and $Z$ by $Y\geq_{st}Z$ if $\mathbb{E}h(Y)\geq\mathbb{E}h(Z)$ for all increasing functions $h$.  Some consequences of this for bounding hitting times and strong stationary times will be considered following the proof of the theorem.
\begin{theorem} \label{thm:ord}
Let $W_j=\inf\{t\geq0 : X_t=j\}$, where we assume that $X_0\sim\pi$.  Let $T^{(j)}$ be a strong stationary time (independent of $W_j$) for the Markov chain $X$ with initial distribution $X_0\sim\pi^{(j)}$.  Let $U=T^{(j)}_1+\cdots+T^{(j)}_N$, where $N\sim\mbox{Geom}(\pi_j)$ and $T^{(j)},T^{(j)}_1,T^{(j)}_2,\ldots$ are IID.  Then $U\geq_{st}W_j$.
\end{theorem}
\begin{proof}
Following equation (3.3) of Daly \cite{d10}, we write 
\begin{equation} \label{eq:stein}
\mathbb{E}h(W_j)-\mathbb{E}h(U)=(1-\pi_j)\mathbb{E}\left[f_h(W_j+T^{(j)})-f_h(Y+T^{(j)})\right]\,,
\end{equation}
where $Y$ is the random variable constructed in the proof of Theorem \ref{thm:TVbd}, $f_h(0)=0$, and 
\begin{equation} \label{eq:stein_sol}
f_h(x)=-\mathbb{E}\left[\sum_{r=0}^\infty(1-\pi_j)^r\{h(V_r)-\mathbb{E}h(U)\}\Bigg|V_0=x\right]\,,
\end{equation}
for $x>0$,where $V_r=V_0+T^{(j)}_1+\cdots+T^{(j)}_r$.  From (\ref{eq:stein_sol}), we can check that if $h$ is increasing, then $f_h$ is decreasing.  Since we have constructed $Y$ in the proof of Theorem \ref{thm:TVbd} in such a way that $Y\leq W_j$ almost surely, using (\ref{eq:stein}) then gives us the desired inequality.
\end{proof}
\begin{corollary} \label{cor:ord}
Let $W_j$ and $T^{(j)}$ be as in Theorem \ref{thm:ord}.  Then
\begin{enumerate}
\item[(i)]
\[
\mathbb{E}W_j\leq\frac{1-\pi_j}{\pi_j}\mathbb{E}T^{(j)}\leq\frac{1}{\pi_j}\sum_{l\not=j}\pi_l\mathbb{E}T_{(l)}\,,
\]
where $T_{(l)}$ is a strong stationary time for $X$ with initial distribution $X_0\sim\delta_l$.
\item[(ii)] If $\theta>0$ is such that $(1-\pi_j)\mathbb{E}e^{\theta T^{(j)}}<1$, then
\[
\mathbb{E}e^{\theta W_j}\leq\frac{\pi_j}{1-(1-\pi_j)\mathbb{E}e^{\theta T^{(j)}}}\,.
\]
\end{enumerate}
\end{corollary} 
\begin{proof}
The first part of (i) and (ii) follow immediately from the standard choices $h(x)=x$ and $h(x)=e^{\theta x}$ (for $\theta>0$), respectively, in Theorem \ref{thm:ord}.  Note that these inequalities are sharp, given our previous comments about cases where $W_j$ is equal in distribution to a geometric sum.

For the second inequality in (i), we apply the first inequality and choose $T^{(j)}$ to be the fastest strong stationary time.  With this choice, using Lemma \ref{lem:sep},
\begin{multline*}
\mathbb{P}(T^{(j)}>t)=\sup_{i\in S}\left(1-\frac{\sum_{l\not=j}\pi_lP^t(l,i)}{\pi_i(1-\pi_j)}\right)=\frac{1}{1-\pi_j}\sup_{i\in S}\left(\sum_{l\not=j}\left[\pi_l-\frac{\pi_lP^t(l,i)}{\pi_i}\right]\right)\\
\leq\frac{1}{1-\pi_j}\sum_{l\not=j}\pi_l\sup_{i\in S}\left(1-\frac{P^t(l,i)}{\pi_i}\right)\leq\frac{1}{1-\pi_j}\sum_{l\not=j}\pi_l\mathbb{P}(T_{(l)}>t)\,.
\end{multline*}
The desired bound follows by summing each side of the above inequality.
\end{proof}
\begin{remark}
\emph{We could combine Corollary \ref{cor:ord} with a stochastic lower bound on $W_j$ to derive inequalities for the strong stationary time $T^{(j)}$ defined in Theorem \ref{thm:ord}.  For example, noting that (when constructing $W_j$ with $X_0\sim\pi$), from state $s\not=j$ it takes at least a geometrically distributed number of steps (with parameter $1-P(s,s)$) before the chain reaches state $j$, we have that $W_j$ is stochastically larger than $I(\xi=0)(1+\eta)$, where $\xi\sim\mbox{Be}(\pi_j)$ and $\eta|X_0\sim\mbox{Geom}(1-P(X_0,X_0))$ (with $X_0\sim\pi^{(j)}$) are independent of all else.  Since the expectation of this random variable is a lower bound for the expectation of $W_j$, we may combine this observation with Corollary \ref{cor:ord}(i) to obtain a lower bound for the expectation of the strong stationary time $T^{(j)}$ in that result:
\[
\mathbb{E}T^{(j)}\geq\pi_j+\frac{\pi_j}{1-\pi_j}\sum_{s\not=j}\frac{\pi_sP(s,s)}{1-P(s,s)}\,.
\]
An analogous lower bound could also be derived for the moment generating function of $T^{(j)}$ based on Corollary \ref{cor:ord}(ii).}
\end{remark}

Defining the average hitting time $\overline{w}=\sum_{j\in S}\pi_j\mathbb{E}W_j$, we note that Corollary \ref{cor:ord}(i) gives the bound
\[
\overline{w}\leq\sum_{j\in S}\sum_{l\not=j}\pi_l\mathbb{E}T_{(l)}\leq(|S|-1)t^*\,,
\]
where $t^*=\sup_{l\in S}\mathbb{E}T_{(l)}$ is the worst-case expected strong stationary time.  As a consequence of this, we note that $(|S|-1)t^*/n$ serves as an upper bound on the total variation distance between the ergodic average of the Markov chain $X$ up to time $n$ and the stationary distribution $\pi$; see Corollary 3 of Roberts and Rosenthal \cite{rr97}.  

We conclude this section by noting that in some cases it is possible to identify the state which achieves the worst case expected strong stationary time, as in the following.
\begin{proposition}
Let $X$ be a reversible Markov chain with transition matrix $P$, and let the state $m\in S$ be such that $\pi_m\geq\pi_l$ for all $l\in S$.  Let $T_{(l)}$ be a fastest strong stationary time for $X$ with $X_0\sim\delta_l$.  If the minimum element of $P^t$ occurs in the column corresponding to transitions into state $m$ for each $t\geq1$, then $T_{(m)}\geq_{st}T_{(l)}$ for each $l\in S$.
\end{proposition}
\begin{proof}
Since $X$ is reversible, we have
\[
\frac{P^t(m,k)}{\pi_k}=\frac{P^t(k,m)}{\pi_m}\leq\frac{P^t(k,m)}{\pi_l}\,,
\]
for each $k,l\in S$ and $t\geq1$.  Also, by assumption, $\inf_kP^t(k,m)\leq\inf_kP^t(k,l)$ for any $l\in S$.  Using Lemma \ref{lem:sep}, applying these two inequalities, and using reversibility of $X$, we have
\begin{multline*}
\mathbb{P}(T_{(m)}>t)=1-\inf_k\left(\frac{P^t(m,k)}{\pi_k}\right)\geq1-\inf_k\left(\frac{P^t(k,m)}{\pi_l}\right)\geq1-\inf_k\left(\frac{P^t(k,l)}{\pi_l}\right)\\
=1-\inf_k\left(\frac{P^t(l,k)}{\pi_k}\right)=\mathbb{P}(T_{(l)}>t)\,,
\end{multline*}
for any $l\in S$.
\end{proof}  

\section{Greedy construction of $T^{(j)}$} \label{sec:greedy}

In Section \ref{sec:approx}, we have considered the approximation of the hitting time $W_j$ by a geometric sum $U$ whose compounding distribution is that of $T^{(j)}$, a strong stationary time for $X$ with $X_0\sim\pi^{(j)}$.  We have already given some bounds which may be useful in the approximation of $T^{(j)}$; we now comment on the explicit construction of this random variable using the greedy approach of Section 3.4 of \cite{df90}.  Although this approach offers no guarantee of constructing the fastest strong stationary time, which is the one which minimizes the upper bound of Theorem \ref{thm:TVbd}, Diaconis and Fill \cite{df90} note that strong stationary times constructed in this way do often achieve this in the examples they consider.  We note this construction here since the choice $X_0\sim\pi^{(j)}$ gives the constructed $T^{(j)}$ a particularly simple structure, which may lead to explicit evaluation of the distribution of this choice of $T^{(j)}$.

Following the work of \cite{df90}, we construct $T^{(j)}$ as the first passage time of a dual Markov chain $X^*$ with state space $S^*=\{A\subseteq S:A\not=\emptyset\}$ to the state $S\in S^*$.  We specify the Markov chain $X^*$ through its initial distribution $\nu^{(j)}$ and transition matrix $P^*$.  These are related to $\pi^{(j)}$ and $P$ through an intertwining matrix $\Lambda$:
\[
\pi^{(j)}=\nu^{(j)}\Lambda\,,\qquad\mbox{and}\qquad\Lambda P=P^*\Lambda\,.
\]  
Letting $S_j=S\setminus\{j\}\in S^*$, the greedy construction of \cite{df90} gives that $X^*$ is initialised in state $S_j$ with probability 1, that is, $\nu^{(j)}=\delta_{S_j}$.  Hence, the row of $\Lambda$ corresponding to $S_j$ contains the vector $\pi^{(j)}$, and we may place the vector $\pi$ along every other row of $\Lambda$.  With this choice, it is straightforward to check that for any $A\in S^*$ with $A\not=S_j$, the row of $\Lambda P$ corresponding to $A$ also consists of the vector $\pi$, and hence the row of the transition matrix $P^*$ corresponding to $A$ consists of a 1 in the column corresponding to $S$, and 0s elsewhere.  It remains only to calculate the entries $P^*(S_j,A)$ for $A\in S^*$, i.e., the row of $P^*$ corresponding to $S_j$.

Using the above construction of $\Lambda$, we calculate that, for each $l\in S$,
\[
(\Lambda P)(S_j,l)=\frac{1}{1-\pi_j}\left[\pi_l-\pi_jP(j,l)\right]\,.
\]
This gives us the following construction of $P^*(S_j,A)$, for $A\in S^*$: Begin by defining $A_0=S$ and $Q_0(S_j,l)=\frac{1}{1-\pi_j}[\pi_l-\pi_jP(j,l)]$.  Then, having defined $A_{r-1}$ and $Q_{r-1}$ (for $r\geq1$), let 
\[
c_r=\min\left\{\frac{Q_{r-1}(S_j,l)}{\pi_l}: l\in A_{r-1}, Q_{r-1}(S_j,l)>0\right\}\,.
\]
If the set over which this minimum is taken is empty, stop and set $z=r-1$.  Otherwise, define
\[
A_r=\left\{l\in A_{r-1}: \frac{Q_{r-1}(S_j,l)}{\pi_l}\geq c_r\right\}\not=\emptyset
\]
and $Q_r(S_j,\cdot)=Q_{r-1}(S_j,\cdot)-c_r\pi_\cdot$ on $A_r$, and continue this recursive procedure.  On termination of this procedure, we have a sequence of elements of $S^*$, $A_z\subseteq A_{z-1}\subseteq\cdots\subseteq A_1$. We define $P^*(S_j,A_r)=c_r\sum_{l\in A_r}\pi_l$ for each $A_r$, and $P^*(S_j,A)=0$ for any $A\in S^*$ which is not equal to one of the $A_r$.

This construction gives a particularly simple Markov chain $X^*$: We start deterministically in the state $S_j$, and from every state apart from $S_j$ we may jump only to the absorbing state $S$.  Considering the first transition made by $X^*$ allows us to calculate $\mathbb{E}T^{(j)}$ explicitly, since it is clear that
\[
\mathbb{E}T^{(j)}=1+\mathbb{P}(X_1^*\not\in\{S,S_j\})+\mathbb{P}(X_1^*=S_j)\mathbb{E}T^{(j)}\,.
\] 
We hence obtain the following.
\begin{proposition}
For $T^{(j)}$ constructed as above, 
\[
\mathbb{E}T^{(j)}=1+\frac{1-P^*(S_j,S)}{1-P^*(S_j,S_j)}\,.
\]
\end{proposition}  
\begin{example}
\emph{Suppose, for simplicity, that
\begin{equation} \label{eq:greedy}
\pi_l\not=\pi_jP(j,l)
\end{equation}
for each $l\in S$, and that the minimum $\min_{l\in S}\left(\frac{\pi_l-\pi_jP(j,l)}{\pi_l}\right)$ is uniquely achieved at $l=j$ (i.e., the choice $l=j$ achieves this minimum, and there is no other state $k\in S$ to do so).  This final assumption is along the lines of that made in Theorem \ref{thm:GeomSum}, though is considerably weaker (since we look only at the one-step transition probabilities here).}

\emph{Following the greedy construction outlined above, under these assumptions we have that $c_1=\frac{1-P(j,j)}{1-\pi_j}$ and $A_1=S$.  Then
\[
c_2=\min\left\{\frac{\pi_l-\pi_jP(j,l)}{\pi_l(1-\pi_j)}-\frac{1-P(j,j)}{1-\pi_j}: l\in S_j\right\}\,,
\]  
and $A_2=S_j$.  The procedure continues, but note that we have now determined the entries $P^*(S_j,S)$ and $P^*(S_j,S_j)$ of the transition matrix $P^*$ which (thanks to the simple structure of the Markov chain $X^*$) are all that are required for calculations.  Letting $\alpha=\frac{1-P(j,j)}{1-\pi_j}$ and $\beta=(1-\pi_j)c_2$, we may then use the underlying structure of $X^*$ to calculate that $\mathbb{P}(T^{(j)}=1)=\alpha$ and $\mathbb{P}(T^{(j)}=t)=(1-\alpha)(1-\beta)\beta^{t-2}$ for $t\geq2$.}
\end{example}  
\begin{example}
\emph{Suppose, as in the previous example, that (\ref{eq:greedy}) holds for each $l\in S$, but now assume that there is some $l^*\not=j$ such that $\min_{l\in S}\left(\frac{\pi_l-\pi_jP(j,l)}{\pi_l}\right)=\frac{\pi_{l^*}-\pi_jP(j,l^*)}{\pi_{l^*}}$.  As above we have $A_1=S$, but now we know that $l^*\not\in A_2$ and so $S_j\not\subseteq A_2$.  Since the sets $A_r$ are decreasing in $r$, we thus have that $S_j\not=A_r$ for any $r$, and so $P^*(S_j,S_j)=0$.  Letting $\gamma=P^*(S_j,S)$, we therefore have that $\mathbb{P}(T^{(j)}=1)=\gamma=1-\mathbb{P}(T^{(j)}=2)$.}
\end{example}

\vspace{12pt}\noindent{\bf Acknowledgements:} The authors thanks Jim Fill for insightful discussions which began this work, and the organisers of Stochastic Models VI in June 2018, where those discussions took place.  Thanks are also due to Robert Gaunt and Xiong Jin for the observations in Remark \ref{rem:bound}.   


\begin{thebibliography}{99}
\bibitem{ad87} D.~Aldous and P.~Diaconis (1987). Strong uniform times and finite random walks. \emph{Adv. in Appl. Math.} {\bf 8}, 69--97.

\bibitem{af} D.~J.~Aldous and J.~A.~Fill. \emph{Reversible Markov Chains and Random Walks on Graphs}. Unfinished monograph, available at \url{https://www.stat.berkeley.edu/~aldous/RWG/book.html}. 

\bibitem{d10} F.~Daly (2010). Stein's method for compound geometric approximation.  \emph{J. Appl. Probab.} {\bf 47}, 146--156.

\bibitem{df90} P.~Diaconis and J.~A.~Fill (1990). Strong stationary times via a new form of duality. \emph{Ann. Probab.} {\bf 18}, 1483--1522.

\bibitem{f09} J.~A.~Fill (2009). On hitting times and fastest strong stationary times for skip-free and more general chains. \emph{J. Theoret. Probab.} {\bf 22}, 587--600.

\bibitem{fl14} J.~A.~Fill and V.~Lyzinski (2014). Hitting times and interlacing eigenvalues: a stochastic approach using intertwinings. \emph{J. Theoret. Probab.} {\bf 27}, 954--981.

\bibitem{lpw09} D.~A.~Levin, Y.~Peres and E.~L.~Wilmer (2009). \emph{Markov Chains and Mixing Times}. American Mathematical Society, Providence, Rhode Island. 

\bibitem{ls12} P.~Lorek and R.~Szekli (2012). Strong stationary duality for M\"obius monotone Markov chains. \emph{Queueing Syst.} {\bf 71}, 79--95.

\bibitem{rr97} G.~O.~Roberts and J.~S.~Rosenthal (1997). Shift-coupling and convergence rates of ergodic averages. \emph{Comm. Statist. Stochastic Models} {\bf 13}(1), 147-–165.

\end{thebibliography}
\end{document}